\documentclass[11pt]{amsart}
\usepackage{tikz-cd}
\usepackage{amscd}
\usepackage{amsmath, amssymb, comment}
\usepackage{amsfonts}
\usepackage{enumerate}
\usepackage{color}
\usepackage{url}
\newcommand{\de}{\partial}

\newcommand{\ov}[1]{\overline{#1}}

\newcommand{\ti}[1]{\tilde{#1}}

\newcommand{\vol}{\mathrm{Vol}}

\renewcommand{\leq}{\leqslant}
\renewcommand{\geq}{\geqslant}

\newcommand{\be}{\begin{equation}}
\newcommand{\ee}{\end{equation}}

\newcommand{\R}{\mathbb{R}}
\newcommand{\C}{\mathbb{C}}

\renewcommand{\H}{\mathcal{H}}
\renewcommand{\O}{\mathcal{O}}
\newcommand{\E}{\mathcal{E}}
\newcommand\ord{\mathrm{ord}}
\newcommand{\dist}{\mathrm{dist}}

\makeatletter
\@namedef{subjclassname@2020}{%
  \textup{2020} Mathematics Subject Classification}
\makeatother

\begin{document}
\newcounter{remark}
\newcounter{theor}
\setcounter{theor}{1}
\newtheorem{claim}{Claim}
\newtheorem{theorem}{Theorem}[section]
\newtheorem{lemma}[theorem]{Lemma}
\newtheorem{corollary}[theorem]{Corollary}
\newtheorem{conjecture}[theorem]{Conjecture}
\newtheorem{proposition}[theorem]{Proposition}
\newtheorem{question}{Question}[section]
\newtheorem{definition}[theorem]{Definition}
\newtheorem{remark}[theorem]{Remark}

\numberwithin{equation}{section}

\title{The rigidity of dimension estimate for holomorphic functions on K\"ahler manifolds}

\author{Jianchun Chu}
\address[Jianchun Chu]{School of Mathematical Sciences, Peking University, Yiheyuan Road 5, Beijing 100871, People's Republic of China}
\email{jianchunchu@math.pku.edu.cn}

\author{Jie Deng}
\address[Jie Deng]{School of Mathematical Sciences, Peking University, Yiheyuan Road 5, Beijing 100871, People's Republic of China}
\email{dj0401@stu.pku.edu.cn}

\author{Zihang Hao}
\address[Zihang Hao]{School of Mathematical Sciences, Peking University, Yiheyuan Road 5, Beijing 100871, People's Republic of China}
\email{2301110014@pku.edu.cn}

\author{Jian Li}
\address[Jian Li]{School of Mathematical Sciences, Peking University, Yiheyuan Road 5, Beijing 100871, People's Republic of China}
\email{jianli25@stu.pku.edu.cn}

\begin{abstract}
In this paper, we obtain the optimal rigidity of dimension estimate for holomorphic functions with polynomial growth on K\"ahler manifolds with non-negative holomorphic bisectional curvature. There is a specific gap between the largest and the second largest dimension. We also determine the optimal dimension that ensures the maximal volume growth which implies the manifold is biholomorphic to the complex Euclidean space.
\end{abstract}


\maketitle

\markboth{Jianchun Chu, Jie Deng, Zihang Hao, and Jian Li}{The rigidity of dimension estimate for holomorphic functions on K\"ahler manifolds}

\section{Introduction}\label{introduction}

The classical uniformization theorem states that a simply-connected Riemann surface is conformally equivalent to the Riemman sphere $\mathbb{CP}^{1}$, the complex plane $\C$, or the open unit disc $\mathbb{D}$. In particular, when the Riemann surface has positive sectional curvature, the following holds.
\begin{enumerate}\setlength{\itemsep}{1mm}
\item[(i)] Each compact Riemann surface with positive curvature is conformally equivalent to the Riemman sphere $\mathbb{CP}^{1}$.
\item[(ii)] Each non-compact complete Riemann surface with positive curvature is conformally equivalent to the complex plane $\C$.
\end{enumerate}
Since each Riemann surface is an one-dimensional K\"ahler manifold, then a natural and fundamental question in K\"ahler geometry is to generalize the above classical results to higher dimensions. The following is the definition of holomorphic bisectional curvature, which is the K\"ahler analog of Riemannian sectional curvature.

\begin{definition}
Let $(M,\omega)$ be an $n$-dimensional K\"ahler manifold and $K$ be a constant. We say that the holomorphic bisectional curvature is greater than or equal to $K$ if
\[
R(u,\ov{u},v,\ov{v}) \geq K\big(|u|^{2}|v|^{2}+|\langle u,\ov{v}\rangle|^{2}\big)
\]
for any $(1,0)$-vectors $u$ and $v$.
\end{definition}

As natural generalizations of (i) and (ii), Frankel \cite{Frankel61} and Yau \cite{Yau89} proposed the following conjectures.

\begin{conjecture}[Frankel's Conjecture \cite{Frankel61}]
Let $(M,\omega)$ be a compact $n$-dimensional K\"ahler manifold with positive holomorphic bisectional curvature. Then $M$ is biholomorphic to $\mathbb{CP}^{n}$.
\end{conjecture}

\begin{conjecture}[Yau's Uniformization Conjecture \cite{Yau89}]\label{Yau uniformization conjecture}
Let $(M,\omega)$ be a complete non-compact $n$-dimensional K\"ahler manifold with positive holomorphic bisectional curvature. Then $M$ is biholomorphic to $\C^{n}$.
\end{conjecture}

Frankel's conjecture was resolved independently by Siu-Yau \cite{SY80} via differential geometry and by Mori \cite{Mori79} via algebraic geometry (more general Hartshorne conjecture was proved, see \cite{Hartshorne70}). As a natural continuation, Howard-Smyth-Wu \cite{HSW81} and Mok \cite{Mok88} established the structure theorem for compact K\"ahler manifolds with non-negative holomorphic bisectional curvature: the universal cover of a compact K\"ahler manifold with non-negative holomorphic bisectional curvature is biholomorphically isometric to the product of complex Euclidean space with compact K\"ahler manifolds of certain type.

For Yau's uniformization conjecture, although it is still open in general, the research along this direction has produced many significant results and methods. In \cite{SY77}, Siu-Yau initiated a program of constructing embedding of a complete K\"ahler manifold into the complex Euclidean space by using holomorphic functions with polynomial growth (see Definition \ref{defs O and Od} below). Such strategy was applied in the work of Mok-Siu-Yau \cite{MSY81} and Mok \cite{Mok84}, where Conjecture \ref{Yau uniformization conjecture} was solved under the maximal volume growth and some curvature decay assumptions. Via K\"ahler-Ricci flow, Chau-Tam \cite{CT06} confirmed Conjecture \ref{Yau uniformization conjecture} under the maximal volume growth and bounded curvature assumptions. Later, only assuming the maximal volume growth, Conjecture \ref{Yau uniformization conjecture} was solved independently by Liu \cite{Liu19} via Gromov-Hausdorff convergence theory and by Lee-Tam \cite{LT20} via K\"ahler-Ricci flow. Here we mention that the above list is far from complete. There are many important contributions by various authors on Conjecture \ref{Yau uniformization conjecture}. We refer the reader to surveys \cite{CT08,CZ24} and the references therein.

\begin{definition}\label{defs O and Od}
Let $(M,\omega)$ be a complete $n$-dimensional K\"ahler manifold. Define
\[
\O(M) = \{f: \text{$f$ is a holomorphic function on $M$}\}.
\]
For $p_{0}\in M$ and $d\geq0$, define
\[
\O_{d}(M) = \{f\in\O(M): |f(p)| \leq C\big(1+r(p_{0},p)\big)^{d} \ \text{for some constant $C$}\},
\]
where $r(p_{0},p)$ denotes the distance between $p_{0}$ and $p$. If $f\in\O_{d}(M)$ for some $d\geq0$, then we say that $f$ is a holomorphic function with polynomial growth.
\end{definition}

For Conjecture \ref{Yau uniformization conjecture}, motivated by the above mentioned strategy of Siu-Yau \cite{SY77}, Yau \cite{Yau90} proposed the study of the space $\O_{d}(M)$ when $(M,\omega)$ has non-negative holomorphic bisectional curvature. It was expected that the following dimension estimate holds:
\[
\dim\mathcal{O}_{d}(M) \leq \dim\mathcal{O}_{d}(\mathbb{C}^{n})
\]
and the equality holds if and only if $(M,\omega)$ is biholomorphically isometric to $(\mathbb{C}^{n},\omega_{\C^{n}})$, where $\omega_{\C^{n}}$ denotes the standard metric on $\C^{n}$.

This expectation was confirmed by Ni \cite{Ni04} under the maximal volume growth assumption, i.e. there exist point $p_{0}\in M$ and constant $c>0$ such that
\[
\vol(B_{r}(p_{0})) \geq c r^{2n} \ \ \text{for any $r>0$}.
\]
Based on the work of \cite{Ni04}, Chen-Fu-Yin-Zhu \cite{CFYZ06} removed the maximal volume growth assumption and then proved this expectation completely. When $(M,\omega)$ is not biholomorphically isometric to $(\mathbb{C}^{n},\omega_{\C^{n}})$, some sharp improved dimension estimates were also established in \cite{Ni04,CFYZ06}. Later, Liu \cite{Liu16a} gave an alternative proof of the above expectation under weaker curvature assumption (see Theorem \ref{Liu result}).

\begin{theorem}[Ni \cite{Ni04}, Chen-Fu-Yin-Zhu \cite{CFYZ06}, Liu \cite{Liu16a}]\label{dimension thm}
Let $(M,\omega)$ be a complete $n$-dimensional K\"ahler manifold with non-negative holomorphic bisectional curvature. For any integer $d\geq1$,
\begin{equation}\label{dimension estimate}
\dim\mathcal{O}_{d}(M) \leq \dim\mathcal{O}_{d}(\mathbb{C}^{n}),
\end{equation}
and the equality holds if and only if $(M,\omega)$ is biholomorphically isometric to $(\mathbb{C}^{n},\omega_{\C^{n}})$.
\end{theorem}

\subsection{First rigidity}

The dimension estimate \eqref{dimension estimate} in Theorem \ref{dimension thm} shows that, among all complete $n$-dimensional K\"ahler manifolds with non-negative holomorphic bisectional curvature, the largest dimension $\dim\mathcal{O}_{d}(M)$ can reach is
\[
\dim\mathcal{O}_{d}(\mathbb{C}^{n}) = \binom{n+d}{d}.
\]
On the other hand, in Section \ref{examples}, we show that there exist infinitely many complete K\"ahler metrics $\omega_{\lambda}$ on $\C^{n}$ such that $(\C^{n},\omega_{\lambda})$ has non-negative holomorphic bisectional curvature and satisfies
\[
\dim\mathcal{O}_{d}(\C^{n},\omega_{\lambda}) = \dim\mathcal{O}_{d}(\mathbb{C}^{n})-\binom{n+d-2}{d-1}.
\]
For convenience, we set
\[
N_{1} = \dim\mathcal{O}_{d}(\mathbb{C}^{n}), \ \
N_{2} = \dim\mathcal{O}_{d}(\mathbb{C}^{n})-\binom{n+d-2}{d-1}.
\]
It is natural to ask: does there exist K\"ahler manifold $(M,\omega)$ with non-negative holomorphic bisectional curvature such that
\[
N_{2}+1 \leq \dim\mathcal{O}_{d}(M) \leq N_{1}-1?
\]

Our first result shows that the answer is negative. Equivalently, for the space $\O_{d}(M)$, the dimension has a gap phenomenon and second largest dimension is exactly $N_{2}$.

\begin{theorem}[First Rigidity]\label{first rigidity}
Let $(M,\omega)$ be a complete $n$-dimensional K\"ahler manifold with non-negative holomorphic bisectional curvature. If there exists an integer $d\geq1$ such that
\[
\dim\mathcal{O}_{d}(M) \geq N_{2}+1,
\]
then $(M,\omega)$ is biholomorphically isometric to $(\mathbb{C}^{n},\omega_{\C^{n}})$.
\end{theorem}

\subsection{Second rigidity}

The next natural question is: when $\dim\O_{d}(M)\leq N_{2}$, what can we say about $(M,\omega)$? Due to examples $(\C^{n},\omega_{\lambda})$ in Section \ref{examples}, we can not expect the rigidity of the metric. Instead, we aim to determine the optimal constant $N_{3}$ such that the dimension condition
\[
\dim\mathcal{O}_{d}(M) \geq N_{3}+1
\]
ensures that $M$ is biholomorphic to $\mathbb{C}^{n}$. Consider the product manifold:
\[
(M_{0},\omega_{0}) = (\C^{n-1}\times\mathbb{CP}^{1},\omega_{\C^{n-1}}+\omega_{\mathbb{CP}^{1}}),
\]
where $\omega_{\C^{n-1}}$ is the standard Euclidean metric on $\C^{n-1}$ and $\omega_{\mathbb{CP}^{1}}$ is the Fubini-Study metric on $\mathbb{CP}^{1}$. It is clear that $(M_{0},\omega_{0})$ has non-negative holomorphic bisectional curvature, but is not biholomorphic to $\mathbb{C}^{n}$. Since $\mathbb{CP}^{1}$ is compact, then any holomorphic function on $M_{0}$ is actually the pull-back of some holomorphic function on $\C^{n-1}$. This implies
\[
\dim\O_{d}(\C^{n-1}\times\mathbb{CP}^{1}) = \dim\O_{d}(\C^{n-1}) = \binom{n-1+d}{d}.
\]
Inspired by this example, we may guess
\[
N_{3} = \binom{n-1+d}{d}.
\]
To prove that $M$ is biholomorphic to $\mathbb{C}^{n}$, thanks to the resolution of Yau's uniformization conjecture under the maximal volume growth assumption (solved independently by Liu \cite{Liu19} and by Lee-Tam \cite{LT20} as mentioned before), it suffices to show that $(M,\omega)$ has maximal volume growth under $\dim\mathcal{O}_{d}(M) \geq N_{3}$. In this direction, Liu \cite{Liu16b} proved the following result.

\begin{theorem}[Liu \cite{Liu16b}]\label{Liu result MVG}
Let $(M,\omega)$ be a complete $n$-dimensional K\"ahler manifold with non-negative holomorphic bisectional curvature. If there exists a constant $c>0$ such that
\[
\dim\mathcal{O}_{d}(M) \geq cd^{n} \, \ \text{for some sufficiently large $d$},
\]
then $(M,\omega)$ has maximal volume growth.
\end{theorem}

We improve Liu's result and obtain the second rigidity.

\begin{theorem}\label{maximal volume growth}
Let $(M,\omega)$ be a complete $n$-dimensional K\"ahler manifold with non-negative holomorphic bisectional curvature. If there exists an integer $d\geq1$ such that
\[
\dim\mathcal{O}_{d}(M) \geq N_{3}+1,
\]
then $(M,\omega)$ has maximal volume growth.
\end{theorem}

\begin{theorem}[Second Rigidity]\label{second rigidity}
Let $(M,\omega)$ be a complete $n$-dimensional K\"ahler manifold with non-negative holomorphic bisectional curvature. If there exists an integer $d\geq1$ such that
\[
\dim\mathcal{O}_{d}(M) \geq N_{3}+1,
\]
then $M$ is biholomorphic to $\mathbb{C}^{n}$.
\end{theorem}

\begin{remark}
By the above example $(M_{0},\omega_{0})$, Theorem \ref{maximal volume growth} also gives the optimal dimension that ensures the maximal volume growth.
\end{remark}

\medskip

The paper is organized as follows: In Section \ref{preliminaries}, we collect some definitions and known results. In Section \ref{first rigidity section} and \ref{maximal volume growth and second rigidity section}, we prove Theorem \ref{first rigidity}, \ref{maximal volume growth} and \ref{second rigidity}. In Section \ref{examples}, some explicit examples are computed, which indicate that our main results are optimal. In Section \ref{questions section}, we pose some questions for further investigation.

\medskip

\subsection*{Acknowledgments}
The authors would like to thank Gang Liu and Kewei Zhang for useful discussions.
The first-named author was partially supported by National Key R\&D Program of China 2024YFA1014800 and 2023YFA1009900, and NSFC grants 12471052 and 12271008. The second-named author was partially supported by National Key R\&D Program of China 2023YFA1009900 and NSFC grant 12271008. The third- and fourth-named authors were partially supported by National Key R\&D Program of China 2023YFA1009900 and 2024YFA1014800, and NSFC grant 12471052.

\bigskip

\section{Preliminaries}\label{preliminaries}

In this section, we collect some definitions and known results that will be used later.

\subsection{Definitions}

\begin{definition}[Vanishing Order]
Let $(M,\omega)$ be a complete $n$-dimensional K\"ahler manifold. For any holomorphic function $f$ on $M$ and $p_{0}\in M$, the vanishing order of $f$ at $p_{0}$ is defined as
\[
\ord_{p_{0}}f = \max\{k\in \mathbb{N}:\nabla^{\alpha}f(x_{0})=0 \ \text{for any $|\alpha|\leq k-1$}\}.
\]
\end{definition}

\begin{definition}[Degree]
Let $(M,\omega)$ be a complete $n$-dimensional K\"ahler manifold and $f$ be a holomorphic function with polynomial growth. The degree of $f$ is defined as
\[
\deg f = \inf\{d\geq0:f\in\O_{d}(M)\}.
\]
\end{definition}

\begin{definition}[Poincar\'e-Siegel Map]
Let $(M,\omega)$ be a complete $n$-dimensional K\"ahler manifold, $p_{0}\in M$ and $z=(z_{1},\ldots,z_{n})$ be holomorphic coordinates centered at $p_{0}$. For any integer $d\geq1$, the Poincar\'e-Siegel map $P_{p_{0},d}:\O_{d}(M)\to\O_{d}(\C^{n})$ is defined as follows.
For any $f\in\O_{d}(M)$, $P_{p_{0},d}(f)$ is the polynomial obtained by truncating the Taylor expansion up to order $d$, i.e.
\[
f(z) = P_{p_{0},d}(f)(z)+O(|z|^{d+1}) \ \ \text{near $p_{0}$}.
\]
\end{definition}

\subsection{Previous results}

\begin{theorem}[Ni \cite{Ni04}, Chen-Fu-Yin-Zhu \cite{CFYZ06}, Liu \cite{Liu16a}]\label{Poincare-Siegel map injective}
Let $(M,\omega)$ be a complete $n$-dimensional K\"ahler manifold with non-negative holomorphic bisectional curvature. For any $p_{0}\in M$ and $f\in\O_{d}(M)$,
\begin{equation}\label{Poincare-Siegel map injective eqn 1}
\ord_{p_{0}}f \leq \deg f.
\end{equation}
In particular, the Poincar\'e-Siegel map $P_{p_{0},d}:\O_{d}(M)\to\O_{d}(\C^{n})$ is injective and so
\[
\dim\mathcal{O}_{d}(M) \leq \dim\mathcal{O}_{d}(\mathbb{C}^{n}).
\]
\end{theorem}

In \cite{CFYZ06}, Chen-Fu-Yin-Zhu actually established the following splitting theorem (see \cite[Section 3]{CFYZ06}).

\begin{theorem}[Chen-Fu-Yin-Zhu \cite{CFYZ06}]\label{splitting}
Let $(X,\omega)$ be a complete simply-connected $n$-dimensional K\"ahler manifold with non-negative holomorphic bisectional curvature. Suppose that there exists a non-trivial holomorphic function $f$ on $X$ such that $\ord_{x_{0}}f=\deg f$ for some $x_{0}\in X$, then $(X,\omega)$ is biholomorphic isometric to
\[
(Y^{n-1} \times \C,\omega_{Y}+\omega_{\C}),
\]
where $(Y^{n-1},\omega_{Y})$ is a complete simply-connected $(n-1)$-dimensional K\"ahler manifold with non-negative holomorphic bisectional curvature.
\end{theorem}

\subsection{Lifted function}

\begin{lemma}\label{lift}
Let $(M,\omega)$ be a complete $n$-dimensional K\"ahler manifold, $(\ti{M},\ti{\omega})$ be the universal cover and $\pi:\ti{M}\to M$ be the projection map. For any function $f$ on $M$, its lifted function on $\ti{M}$ is defined as $\ti{f}=f\circ\pi$.
\begin{enumerate}\setlength{\itemsep}{1mm}
\item[(i)] $\ti{f}$ is $\pi_{1}(M)$-invariant, i.e. $\ti{f}=\ti{f}\circ\sigma$ for any deck transformation $\sigma\in\pi_{1}(M)$.
\item[(ii)] If $f\in\O_{d}(M)$, then $\ti{f}\in\O_{d}(\ti{M})$.
\end{enumerate}
\end{lemma}

\begin{proof}
(i) follows from $\pi\circ\sigma=\pi$ for any deck transformation $\sigma\in\pi_{1}(M)$. For (ii), Since $f\in\O_{d}(M)$, then
\[
|f(p)| \leq C\big(1+r(p_{0},p)\big)^{d}
\]
for some constant $C$. Fix $\ti{p}_{0}\in\pi^{-1}(p_{0})$. For any $\ti{p}\in\ti{M}$, we see that
\[
|\ti{f}(\ti{p})| = |f(\pi(\ti{p}))| \leq C\big(1+r(p_{0},\pi(\ti{p}))\big)^{d}
= C\big(1+r(\pi(\ti{p}_{0}),\pi(\ti{p}))\big)^{d}.
\]
Since
\[
r(\pi(\ti{p}_{0}),\pi(\ti{p})) \leq r(\ti{p}_{0},\ti{p}),
\]
then
\[
|\ti{f}(\ti{p})| \leq C\big(1+r(\ti{p}_{0},\ti{p})\big)^{d},
\]
which implies $\ti{f}\in\O_{d}(\ti{M})$.
\end{proof}

\subsection{Classical results on $\C^{n}$}

\begin{theorem}[Classical Liouville Theorem on $\C^{n}$]\label{classical Liouville theorem}
If $f\in\O_{d}(\C^{n})$ for some integer $d\geq0$, then $f$ is a polynomial with degree at most $d$.
\end{theorem}

\begin{theorem}[Biholomorphic Isometry of $\C^{n}$]\label{biholomorphic isometry group of Cn}
If $\sigma$ is a biholomorphic isometry of $\C^{n}$, then it can be written as $\sigma(z)=Az+c$ where $A$ is a unitary matrix and $c\in\C^{n}$.
\end{theorem}

\subsection{Some spaces}\label{Some spaces}
Let $p,q,d\geq0$ be three integers. Write $z\in\C^{p}$ and $x\in\R^{q}$. Define three spaces by
\[
\mathcal{H}_{d}(\R^{q}) = \{\text{complex harmonic polynomials on $\R^{q}$ with degree at most $d$}\}, \\[0.6mm]
\]
\[
\mathcal{P}_{d}(\C^{p}\times\R^{q}) = \{\text{complex polynomials of $(z,x)$ with degree at most $d$}\}, \\[1.2mm]
\]
\[
\mathcal{E}_{d}(\C^{p}\times\R^{q})
= \{f\in\mathcal{P}_{d}(\C^{p}\times\R^{q}):\text{$f(z,\cdot)\in\mathcal{H}_{d}(\R^{q})$ for any $z\in\C^{p}$}\}.
\]

\begin{lemma}\label{dimension of E}
The dimension of space $\mathcal{E}_{d}(\C^{p}\times\R^{q})$ is
\[
\begin{split}
\dim\mathcal{E}_{d}(\C^{p}\times\R^{q})
= {} & \binom{p+q+d}{d}-\binom{p+q+d-2}{d-2} \\[1mm]
= {} & \binom{p+q+d}{p+q}-\binom{p+q+d-2}{p+q}.
\end{split}
\]
\end{lemma}

\begin{proof}
It is well-known that
\[
\dim\mathcal{P}_{d}(\C^{p}\times\R^{q}) = \binom{p+q+d}{d}.
\]
When $d=0,1$, we have $\mathcal{E}_{d}(\C^{p}\times\R^{q})=\mathcal{P}_{d}(\C^{p}\times\R^{q})$ and then Lemma \ref{dimension of E} follows. Next we assume that $d\geq2$. Let $\Delta_{x}$ be the Laplacian operator in $\R^{q}$. We claim that
\begin{equation}\label{dimension of E claim}
\text{$\Delta_{x}:\mathcal{P}_{d}(\C^{p}\times\R^{q})\to\mathcal{P}_{d-2}(\C^{p}\times\R^{q})$ is surjective}.
\end{equation}
Give this claim, it is clear that
\[
\begin{split}
\dim\mathcal{E}_{d}(\C^{p}\times\R^{q}) = {} & \dim\mathrm{Ker}\,\Delta_{x} \\[2.6mm]
= {} & \dim\mathcal{P}_{d}(\C^{p}\times\R^{q})-\dim\mathcal{P}_{d-2}(\C^{p}\times\R^{q}) \\[1.2mm]
= {} & \binom{p+q+d}{d}-\binom{p+q+d-2}{d-2},
\end{split}
\]
as required.

\medskip

Now we prove the claim \eqref{dimension of E claim}. Define
\[
T: \mathcal{P}_{d-2}(\C^{p}\times\R^{q}) \to \mathcal{P}_{d-2}(\C^{p}\times\R^{q}), \ \ \
Tf = \Delta_{x}\Big((1-|x|^{2})f\Big).
\]
For any $f\in\mathrm{Ker}\,T$ and $z\in\C^{p}$, we set $f_{z}(x)=f(z,x)$ and so
\[
\Delta_{x}\Big((1-|x|^{2})f_{z}\Big) = 0 \, \ \text{in $\R^{q}$}.
\]
It then follows that $(1-|x|^{2})\cdot\mathrm{Re}f_{z}$ is a real harmonic function on $\R^{q}$ with
\[
(1-|x|^{2})\cdot\mathrm{Re}f_{z} \equiv 0 \, \ \text{on $\de B_{1}$}.
\]
By the maximum principle,
\[
(1-|x|^{2})\cdot\mathrm{Re}f_{z} \equiv 0 \, \ \text{in $B_{1}$}.
\]
Since any harmonic function is analytic, then
\[
(1-|x|^{2})\cdot\mathrm{Re}f_{z} \equiv 0 \, \ \text{in $\R^{q}$},
\]
and so $\mathrm{Re}f_{z}\equiv0$ in $\R^{q}$. Similarly, we have $\mathrm{Im}f_{z}\equiv0$ in $\R^{q}$. Then $f_{z}\equiv0$ for any $z\in\C^{p}$ and so $f\equiv0$. This shows $\mathrm{Ker}\,T=\{0\}$. Combining this with $\dim\mathcal{P}_{d-2}(\C^{p}\times\R^{q})<\infty$, the map $T$ is bijective, which implies \eqref{dimension of E claim}.
\end{proof}

\begin{remark}\label{dimension of H}
When $p=0$ and $q=n$, we have $\E_{d}(\R^{n})=\H_{d}(\R^{n})$ and then Lemma \ref{dimension of E} shows
\[
\dim\H_{d}(\R^{n})
= \binom{n+d}{d}-\binom{n+d-2}{d-2}
= \binom{n+d}{n}-\binom{n+d-2}{n}.
\]
\end{remark}

\begin{remark}\label{dimension upper bound of E and H}
Lemma \ref{dimension of E} and Remark \ref{dimension of H} show that
\begin{enumerate}\setlength{\itemsep}{1.2mm}
\item[(i)] $\dim\H_{d}(\R^{n})\leq C(n)d^{n-1}$ for some constant $C(n)$;
\item[(ii)] $\dim\mathcal{E}_{d}(\C^{p}\times\R^{q}) \leq C(p,q)d^{p+q-1}$ for some constant $C(p,q)$.
\end{enumerate}
\end{remark}

\section{First rigidity}\label{first rigidity section}

In Subsection \ref{splitting theorem subsection} and \ref{simply connectedness subsection}, we will investigate the properties of the underlying manifold when $\dim\O_{d}(M)$ is sufficiently large. The proof of Theorem \ref{first rigidity} (i.e. first rigidity) will be given in Subsection \ref{proof of first rigidity subsection}.

\subsection{Splitting theorem}\label{splitting theorem subsection}

In this subsection, we will establish a splitting theorem, which is a generalization of Theorem \ref{splitting}.

\begin{theorem}\label{splitting Ck}
Let $(X^{n},\omega)=(Y^{n-k}\times\C^{k},\omega_{Y}+\omega_{\C^{k}})$, where $0\leq k\leq n-1$ and $(Y^{n-k},\omega_{Y})$ be a complete simply-connected $(n-k)$-dimensional K\"ahler manifold with non-negative holomorphic bisectional curvature. Suppose that there exists an integer $d\geq1$ such that
\[
\dim\mathcal{O}_{d}(X)\geq \dim\mathcal{O}_{d-1}(\mathbb{C}^{n})+\binom{k+d-1}{d}+1.
\]
Then $(Y^{n-k},\omega_{Y})$ is biholomorphic isometric to
\[
(Z^{n-k-1}\times\C,\omega_{Z}+\omega_{\C}),
\]
where $(Z^{n-k-1},\omega_{Z})$ be a complete simply-connected $(n-k-1)$-dimensional K\"ahler manifold with non-negative holomorphic bisectional curvature. In particular, $(X^{n},\omega)$ is biholomorphic isometric to
\[
(Z^{n-k-1}\times\C^{k+1},\omega_{Z}+\omega_{\C^{k+1}}).
\]
\end{theorem}

\begin{proof}
Fix a point $y_{0}\in Y$, and let $y=(y_{1},\ldots,y_{n-k})$ be the holomorphic coordinates centered at $y_{0}$ and write $z=(z_{1},\ldots,z_{k})\in\C^{k}$. By Theorem \ref{splitting}, it suffices to construct a non-trivial holomorphic function $f_{Y}$ on $Y$ such that $\ord_{y_{0}}f_{Y}=\deg f_{Y}$. We split the argument into five steps.

\bigskip
\noindent
{\bf Step 1.} There exists $f\in\O_{d}(X)$ such that $\ord_{(y_{0},0)}f=d$ and $P_{(y_{0},0),d}(f)(y,z)$ involves $y$.
\bigskip

Suppose that there is no such $f$. Then for any $f\in\O_{d}(X)$, the polynomial $P_{(y_{0},0),d}(f)(y,z)$ can be written as the linear combination of
\begin{itemize}\setlength{\itemsep}{1mm}
\item monomial of $(y_{1},\ldots,y_{n-k},z_{1},\ldots,z_{k})$ with degree at most $d-1$;
\item monomial of $(z_{1},\ldots,z_{k})$ with degree $d$.
\end{itemize}
It follows that the dimension of the image of $P_{(y_{0},0),d}$ satisfies
\[
\dim\mathrm{Im}P_{(y_{0},0),d} \leq \dim\mathcal{O}_{d-1}(\mathbb{C}^{n})+\binom{k+d-1}{d}.
\]
By Theorem \ref{Poincare-Siegel map injective}, $P_{(y_{0},0),d}:\O_{d}(X)\to\O_{d}(\C^{n})$ is injective and so
\[
\dim\mathcal{O}_{d}(X) \leq \dim\mathcal{O}_{d-1}(\mathbb{C}^{n})+\binom{k+d-1}{d},
\]
which contradicts with the assumption on $\dim\mathcal{O}_{d}(X)$.

\bigskip
\noindent
{\bf Step 2.} Decompose $f$.
\bigskip

Let $f$ be the function obtained in Step 1. Since $f\in\O_{d}(X)$, then
\[
|f(y,z)| \leq C\big(1+r(y_{0},y)+|z|\big)^{d}.
\]
This shows that, for any fixed $y\in Y$, $f(y,\cdot)$ is a holomorphic function on $\C^{k}$ with polynomial growth of degree at most $d$. Then by the classical Liouville theorem (i.e. Theorem \ref{classical Liouville theorem}), we know that $f(y,\cdot)$ is actually a polynomial on $\C^{k}$ with degree at most $d$, i.e.
\[
f(y,z) = \sum_{|\alpha|\leq d}f_{\alpha}(y)\cdot z^{\alpha},
\]
where
\[
\alpha = (\alpha_{1},\ldots,\alpha_{k}), \ \
z^{\alpha} = z_{1}^{\alpha_{1}}\cdots z_{k}^{\alpha_{k}}.
\]
Each $f_{\alpha}(y)$ is actually a holomorphic function on $Y$, since it can be written as the partial derivative of $f$:
\[
f_{\alpha}(y) = \frac{1}{\alpha_{1}!\cdots\alpha_{k}!}\cdot\frac{\de^{|\alpha|}f(y,z)}{\de z_{1}^{\alpha_{1}}\cdots\de z_{k}^{\alpha_{k}}}\Bigg|_{z=0}.
\]

\bigskip
\noindent
{\bf Step 3.} $f_{\alpha}(y)\cdot z^{\alpha}\in\O_{d}(X)$ for each $\alpha$.
\bigskip

Write $z=(z_{1},z')$ and $\alpha=(\alpha_{1},\alpha')$ where
\[
z' = (z_{2},\ldots,z_{k}), \ \ \alpha' = (\alpha_{2},\ldots,\alpha_{k}).
\]
By Step 2, we have the decomposition of $f$:
\[
f(y,z) = \sum_{|\alpha|\leq d}f_{\alpha}(y)\cdot z^{\alpha},
\]
where each $f_{\alpha}(y)$ is a holomorphic function on $Y$. We rearrange the above decomposition as
\[
f(y,z_{1},z') = \sum_{\alpha_{1}=0}^{d}\left(\sum_{|\alpha'|\leq d-\alpha_{1}}f_{(\alpha_{1},\alpha')}(y)\cdot (z')^{\alpha'}\right) z_{1}^{\alpha_{1}}.
\]
Then for $j=1,\ldots,d+1$, we have
\[
f(y,jz_{1},z') = \sum_{\alpha_{1}=0}^{d}j^{\alpha_{1}}\left(\sum_{|\alpha'|\leq d-\alpha_{1}}f_{(\alpha_{1},\alpha')}(y)\cdot (z')^{\alpha'}\cdot z_{1}^{\alpha_{1}}\right) .
\]
Since the Vandermonde determinant
\[
\det
\begin{pmatrix}
1^{0} & 1^{1} & \cdots & 1^{d} \\
2^{0} & 2^{1} & \cdots & 2^{d} \\
\vdots & \vdots & \ddots & \vdots \\
(d+1)^{0} & (d+1)^{1} & \cdots & (d+1)^{d}
\end{pmatrix}
\neq 0,
\]
then by Cramer's rule, for each $\alpha_{1}=0,\ldots,d$,
\[
\sum_{|\alpha'|\leq d-\alpha_{1}}f_{(\alpha_{1},\alpha')}(y)\cdot (z')^{\alpha'}\cdot z_{1}^{\alpha_{1}}
\]
can be written as the linear combination of
\[
f(y,jz_{1},z')\in O_{d}(X), \ \ \text{$j=1,\ldots,d+1$}.
\]
It then follows that
\[
\sum_{|\alpha'|\leq d-\alpha_{1}}f_{(\alpha_{1},\alpha')}(y)\cdot (z')^{\alpha'}\cdot z_{1}^{\alpha_{1}} \in \O_{d}(X) \ \ \text{for each $\alpha_{1}=0,\ldots,d$}.
\]
Applying the above argument repeatedly, we obtain $f_{\alpha}(y)\cdot z^{\alpha}\in\O_{d}(X)$ for each $\alpha$.

\bigskip
\noindent
{\bf Step 4.} $f_{\alpha}(y)\in\O_{d-|\alpha|}(Y)$ for each $\alpha$.
\bigskip

If $f_{\alpha}(y)\notin\O_{d-|\alpha|}(Y)$ for some $\alpha$, then there exists a sequence of points $\{\mathbf{y}_{i}\}\subseteq Y$ such that
\[
\lim_{i\to\infty}\frac{|f_{\alpha}(\mathbf{y}_{i})|}{\big(1+r(y_{0},\mathbf{y}_{i})\big)^{d-|\alpha|}} = \infty.
\]
Set $r_{i}=r(y_{0},\mathbf{y}_{i})$ and $\mathbf{z}_{i} = (r_{i},\ldots,r_{i})\in\C^{k}$.
Then $\{\mathbf{z}_{i}\}$ is a sequence of points in $\C^{k}$ with
\[
\mathbf{z}_{i}^{\alpha} = r_{i}^{\alpha_{1}}\cdots r_{i}^{\alpha_{k}} = r_{i}^{|\alpha|}, \ \
|\mathbf{z}_{i}| = \sqrt{k}\cdot r_{i}.
\]
Combining this with $f_{\alpha}(y)\cdot z^{\alpha}\in\O_{d}(X)$ (see Step 3),
\[
C \geq \frac{|f_{\alpha}(\mathbf{y}_{i}) \cdot \mathbf{z}_{i}^{\alpha}|}{\big(1+r(y_{0},\mathbf{y}_{i})+|\mathbf{z}_{i}|\big)^{d}}
= \frac{|f_{\alpha}(\mathbf{y}_{i})|\cdot r_{i}^{|\alpha|}}{\big(1+r_{i}+\sqrt{k}\cdot r_{i}\big)^{d}}
\geq \frac{|f_{\alpha}(\mathbf{y}_{i})|}{C\big(1+r_{i}\big)^{d-|\alpha|}}\to \infty,
\]
which is a contradiction.

\bigskip
\noindent
{\bf Step 5.} Find a non-trivial holomorphic function $f_{Y}$ on $Y$ such that $\ord_{y_{0}}f_{Y}=\deg f_{Y}$.
\bigskip

Recall that $f$ has the following decomposition (see Step 2):
\[
f(y,z) = \sum_{|\alpha|\leq d}f_{\alpha}(y)\cdot z^{\alpha}.
\]
It is clear that
\[
P_{(y_{0},0),d}(f)(y,z) = \sum_{|\alpha|\leq d}P_{y_{0},d-|\alpha|}(f_{\alpha})(y)\cdot z^{\alpha}.
\]
Since $\ord_{(y_{0},0)}f=d$ (see Step 1), then $P_{(y_{0},0),d}(f)(y,z)$ is a homogenous polynomial of $(y,z)$ with degree $d$. This shows that, for each $\alpha$, either $P_{y_{0},d-|\alpha|}(f_{\alpha})(y)\equiv0$, or $P_{y_{0},d-|\alpha|}(f_{\alpha})(y)$ is a homogenous polynomial of $y$ with degree $d-|\alpha|$. Since $P_{(y_{0},0),d}(f)(y,z)$ involves $y$ (see Step 1), then at least one of $P_{y_{0},d-|\alpha|}(f_{\alpha})(y)$ involves $y$ and this $f_{\alpha}(y)$ satisfies
\[
\ord_{y_{0}}f_{\alpha} = d-|\alpha| \geq 1.
\]
Combining $f_{\alpha}(y)\in\O_{d-|\alpha|}(Y)$ (see Step 4) and Theorem \ref{Poincare-Siegel map injective},
\[
d-|\alpha| = \ord_{y_{0}}f_{\alpha} \leq \deg f_{\alpha} \leq d-|\alpha|,
\]
which implies $1\leq\ord_{y_{0}}f_{\alpha}=\deg f_{\alpha}$, and so $f_{\alpha}(y)$ is the required non-trivial holomorphic function on $Y$.
\end{proof}

\begin{corollary}\label{flatness}
Let $(X,\omega)$ be a complete simply-connected $n$-dimensional K\"ahler manifold with non-negative holomorphic bisectional curvature. Suppose that there exists an integer $d\geq1$ such that
\[
\dim\mathcal{O}_{d}(X) \geq N_{2}+1
= \dim\mathcal{O}_{d}(\mathbb{C}^{n})-\binom{n+d-2}{d-1}+1.
\]
Then $(X,\omega)$ is biholomorphically isometric to $(\mathbb{C}^{n},\omega_{\C^{n}})$.
\end{corollary}

\begin{proof}
It is clear that
\[
\begin{split}
N_{2} = {} & \dim\mathcal{O}_{d}(\mathbb{C}^{n})-\binom{n+d-2}{d-1} \\
= {} & \binom{n+d}{d}-\binom{n+d-2}{d-1} \\
= {} & \binom{n+d-1}{d-1}+\binom{n+d-1}{d}-\binom{n+d-2}{d-1} \\
= {} & \dim\mathcal{O}_{d-1}(\mathbb{C}^{n})+\binom{n+d-2}{d}.
\end{split}
\]
Then for each $0\leq k\leq n-1$, \\
\[
\begin{split}
\dim\mathcal{O}_{d}(X)
\geq {} & N_{2}+1 \\[1mm]
= {} & \dim\mathcal{O}_{d-1}(\mathbb{C}^{n})+\binom{n+d-2}{d}+1 \\
\geq {} & \dim\mathcal{O}_{d-1}(\mathbb{C}^{n})+\binom{k+d-1}{d}+1.
\end{split}
\]
Applying Theorem \ref{splitting Ck} to $(X,\omega)$ repeatedly, we complete the proof.
\end{proof}

\subsection{Simply-connectedness}\label{simply connectedness subsection}

In this subsection, we will show that if $(M,\omega)$ is flat and $\dim\mathcal{O}_{d}(M)$ is sufficiently large, then $M$ is simply-connected.

\begin{theorem}\label{simply connectedness}
Let $(M,\omega)$ be a complete $n$-dimensional K\"ahler manifold such that its universal cover $(\ti{M},\ti{\omega})$ is biholomorphically isometric to $(\C^{n},\omega_{\C^{n}})$. If there exists an integer $d\geq1$ such that
\[
\dim\mathcal{O}_{d}(M) \geq \dim\mathcal{O}_{d}(\mathbb{C}^{n})-\binom{n+d-1}{d-1}+1,
\]
then $M$ is simply-connected. In particular, $(M,\omega)$ is biholomorphically isometric to $(\mathbb{C}^{n},\omega_{\C^{n}})$.
\end{theorem}

\begin{proof}
Regard each $\sigma\in\pi_{1}(M)$ as a deck transformation of $\ti{M}$. It suffices to show that $\sigma$ is the identity map. Since $(\ti{M},\ti{\omega})$ is biholomorphically isometric to $(\C^{n},\omega_{\C^{n}})$, then $\sigma$ can be regarded as a biholomorphic isometry of $\C^{n}$. By Theorem \ref{biholomorphic isometry group of Cn}, $\sigma$ can be expressed as
\[
\sigma(z) = Az+c,
\]
where $A$ is a unitary matrix and $c\in\C^{n}$. Let $\lambda_{1},\ldots,\lambda_{k}$ be the eigenvalues of $A$ that are not equal to $1$ and set $\ell=n-k$. After making a unitary transformation, we assume without loss of generality that
\[
\sigma(z) =
\begin{pmatrix}
\Lambda & 0 \\
0  & \mathbf{1}_{\ell}
\end{pmatrix}
z+
\begin{pmatrix}
a \\
b
\end{pmatrix},
\]
where
\[
\Lambda
= \begin{pmatrix}
\lambda_{1} & & \\
 & \ddots & \\
 & & \lambda_{k}
\end{pmatrix},
\ \
\mathbf{1}_{\ell}
= \begin{pmatrix}
1 & & \\
 & \ddots & \\
 & & 1
\end{pmatrix}, \ \
a\in\C^{k}, \ \
b\in\C^{\ell}.
\]
To show $\sigma$ is the identity map, it suffices to show that $b=0$. Indeed, $b=0$ implies that
\[
\begin{pmatrix}
a_{\Lambda}\\
0
\end{pmatrix}
\, \text{is a fixed point of $\sigma$, where} \
a_{\Lambda} =
\begin{pmatrix}
\frac{a_{1}}{1-\lambda_{1}} \\
\vdots \\
\frac{a_{k}}{1-\lambda_{k}}
\end{pmatrix}.
\]
Recalling that $\sigma$ is a deck transformation, then $\sigma$ must be the identity map.

To show $b=0$, we introduce some definitions:
\[
\sigma'(z) =
\begin{pmatrix}
\Lambda & 0 \\
0  & \mathbf{1}_{\ell}
\end{pmatrix}
z+
\begin{pmatrix}
0 \\
b
\end{pmatrix}, \ \
\sigma''(z) =
z+
\begin{pmatrix}
0 \\
b
\end{pmatrix}, \ \
\tau(z) =
z+
\begin{pmatrix}
a_{\Lambda} \\
0
\end{pmatrix}
\]
and
\[
\ti{F} = \{\ti{f} = f\circ\pi:f\in\O_{d}(M)\}.
\]
By Lemma \ref{lift}, $\ti{F}\subseteq\O_{d}(\ti{M})$. Then each $\ti{f}\in\ti{F}$ can be regarded as a polynomial in $\O_{d}(\C^{n})$, and so is an $\sigma$-invariant polynomial with degree at most $d$. We split the argument of $b=0$ into three steps.

\bigskip
\noindent
{\bf Step 1.} For each $\ti{f}\in\ti{F}$, $\ti{f}_{\tau}=\ti{f}\circ\tau$ is $\sigma'$-invariant.
\bigskip

By the definition of $a_{\Lambda}$, it is clear that $a_{\Lambda}=\Lambda a_{\Lambda}+a$ and so
\[
\tau\circ\sigma'(z)
= \begin{pmatrix}
\Lambda & 0 \\
0  & \mathbf{1}_{\ell}
\end{pmatrix}
z+
\begin{pmatrix}
a_{\Lambda} \\
b
\end{pmatrix}
= \begin{pmatrix}
\Lambda & 0 \\
0  & \mathbf{1}_{\ell}
\end{pmatrix}
z+
\begin{pmatrix}
\Lambda a_{\Lambda}+a \\
b
\end{pmatrix}
= \sigma\circ\tau(z).
\]
Since $\ti{f}$ is $\sigma$-invariant, then $\ti{f}\circ\sigma=\ti{f}$ and so
\[
\ti{f}_{\tau}\circ\sigma' = \ti{f}\circ\tau\circ\sigma' = \ti{f}\circ\sigma\circ\tau = \ti{f}\circ\tau = \ti{f}_{\tau},
\]
which implies $\ti{f}_{\tau}$ is $\sigma'$-invariant.

\bigskip
\noindent
{\bf Step 2.} For each $\ti{f}\in\ti{F}$, $\ti{f}_{\tau}$ is $\sigma''$-invariant.
\bigskip

Write $z=(x,y)$ where
\[
x = (x_{1},\ldots,x_{k}), \ \ y = (y_{1},\ldots,y_{\ell}),
\]
and suppose that $\ti{f}_{\tau}$ can be written as
\[
\ti{f}_{\tau}(x,y) = \sum_{|\alpha|+|\beta|\leq d} c_{\alpha,\beta} \cdot x^{\alpha} \cdot y^{\beta}.
\]
where
\[
c_{\alpha,\beta}\in\C, \ \
\alpha = (\alpha_{1},\ldots,\alpha_{k}), \ \
\beta = (\beta_{1},\ldots,\beta_{\ell}).
\]
Since $\ti{f}_{\tau}$ is $\sigma'$-invariant (see Step 1), then
\begin{equation}\label{first rigidity eqn 1}
\sum_{|\alpha|+|\beta|\leq d} c_{\alpha,\beta} \cdot \prod_{i=1}^{k}(\lambda_{i}x_{i})^{\alpha_i} \cdot \prod_{j=1}^{\ell}(y_{j}+b_{j})^{\beta_j}
= \sum_{|\alpha|+|\beta|\leq d} c_{\alpha,\beta} \cdot \prod_{i=1}^{k} x_{i}^{\alpha_i} \cdot \prod_{j=1}^{\ell} y_{j}^{\beta_j}.
\end{equation}
We claim that
\begin{equation}\label{first rigidity claim}
c_{\alpha,\beta}\cdot\prod_{i=1}^{k}\lambda_{i}^{\alpha_i} = c_{\alpha,\beta} \ \, \text{for any $\alpha$ and $\beta$}.
\end{equation}
Given this claim, \eqref{first rigidity eqn 1} becomes
\[
\sum_{|\alpha|+|\beta|\leq d} c_{\alpha,\beta} \cdot \prod_{i=1}^{k}x_{i}^{\alpha_i} \cdot \prod_{j=1}^{\ell}(y_{j}+b_{j})^{\beta_j}
= \sum_{|\alpha|+|\beta|\leq d} c_{\alpha,\beta} \cdot \prod_{i=1}^{k} x_{i}^{\alpha_i} \cdot \prod_{j=1}^{\ell} y_{j}^{\beta_j},
\]
which implies $\ti{f}_{\tau}$ is $\sigma''$-invariant, as required.

Next we prove \eqref{first rigidity claim}. For each fixed $\alpha=(\alpha_{1},\ldots,\alpha_{k})$, by comparing the term containing $\prod_{i=1}^{k} x_{i}^{\alpha_i}$ on both sides of \eqref{first rigidity eqn 1}, we have
\begin{equation}\label{first rigidity eqn 2}
\prod_{i=1}^{k} \lambda_{i}^{\alpha_i}\left(\sum_{|\beta|\leq d-|\alpha|} c_{\alpha,\beta}\cdot\prod_{j=1}^{\ell}(y_{j}+b_{j})^{\beta_j} \right)
= \sum_{|\beta|\leq d-|\alpha|} c_{\alpha,\beta}\cdot\prod_{j=1}^{\ell}y_{j}^{\beta_j}.
\end{equation}
If $c_{\alpha,\beta}=0$ for all $\beta$, then \eqref{first rigidity claim} is trivial. Otherwise, by comparing the coefficient of the leading term of \eqref{first rigidity eqn 2}, we obtain $\prod_{i=1}^{k}\lambda_{i}^{\alpha_i}=1$, which also implies \eqref{first rigidity claim}.

\bigskip
\noindent
{\bf Step 3.} Prove $b=0$.
\bigskip

Define
\[
V = \{\ti{f}_{\tau}=\ti{f}\circ\tau:\ti{f}\in \ti{F}\}.
\]
Then
\[
\dim V = \dim \ti{F} = \dim\O_{d}(M) \geq \dim\mathcal{O}_{d}(\mathbb{C}^{n})-\binom{n+d-1}{d-1}+1.
\]
Each $\ti{f}_{\tau}\in V$ is $\sigma''$-invariant (see Step 2) and $\sigma''$ is a translation of $\C^{n}$. Using Lemma \ref{translation lemma} below, we obtain $b=0$.
\end{proof}

\begin{lemma}\label{translation lemma}
Let $\rho(z)=z+c$ be a translation of $\C^{n}$ and $V\subseteq\O_{d}(\C^{n})$ be a subspace such that every $g\in V$ is $\rho$-invariant. If
\[
\dim V \geq \dim\mathcal{O}_{d}(\mathbb{C}^{n})-\binom{n+d-1}{d-1}+1,
\]
then $c=0$.
\end{lemma}

\begin{proof}
After making a unitary transformation, we assume without loss of generality that
\[
c = (c_{1},0,\ldots,0).
\]
It suffices to show $c_{1}=0$. Define the space $W$ by
\[
W = \mathrm{span}\{z_{1}^{\alpha_{1}}\cdots z_{n}^{\alpha_{n}}\in\O_{d}(\C^{n}):\alpha_{1}\neq0\}.
\]
It then follows that
\[
\dim W = \dim\mathcal{O}_{d}(\mathbb{C}^{n})-\dim\mathcal{O}_{d}(\mathbb{C}^{n-1})
= \binom{n+d}{d}-\binom{n+d-1}{d}.
\]
We compute
\[
\begin{split}
\dim(V\cap W)
\geq {} & \dim W+\dim V-\dim\mathcal{O}_{d}(\mathbb{C}^{n}) \\[1.6mm]
\geq {} & \binom{n+d}{d}-\binom{n+d-1}{d}-\binom{n+d-1}{d-1}+1 \\[1.8mm]
= {} & 1.
\end{split}
\]
Then there exists a non-trivial polynomial $g\in V\cap W$. Write $z=(z_{1},z')$ where $z'=(z_{2},\ldots,z_{n})$. Then $g$ can be written as
\[
g(z_{1},z') = \sum_{j=0}^{m}g_{j}(z') \cdot z_{1}^{j},
\]
where $m\leq d$ and $g_{m}(z')$ is a non-zero polynomial. Since $g\in W$, then $g$ involves $z_{1}$ and so $m\geq1$. Fix $z_{0}'$ such that $g_{m}(z_{0}')\neq0$ and define
\[
P(z_{1}) = g(z_{1},z_{0}')-g(0,z_{0}') = \sum_{j=0}^{m}g_{j}(z_{0}') \cdot z_{1}^{j} - g_{0}(z_{0}').
\]
Then $P(z_{1})$ is a polynomial of $z_{1}$ with degree $m\geq1$. Since $g\in V$, then $g$ is $\rho$-invariant and so
\[
P(z_{1}) = g(z_{1},z_{0}')-g(0,z_{0}') = g(z_{1}+c_{1},z_{0}')-g(0,z_{0}') = P(z_{1}+c_{1}).
\]
It is clear that $P(kc_{1})=0$ for any $k\in\mathbb{Z}$. If $c_{1}\neq0$, then the polynomial $P(z_{1})$ has infinitely many zeros, which is impossible.
\end{proof}

\subsection{Proof of Theorem \ref{first rigidity}}\label{proof of first rigidity subsection}

Now we are ready to prove Theorem \ref{first rigidity}.

\begin{proof}[Proof of Theorem \ref{first rigidity}]
Applying Corollary \ref{flatness} to the universal cover $(\ti{M},\ti{\omega})$, we obtain that $(\ti{M},\ti{\omega})$ is biholomorphically isometric to $(\C^{n},\omega_{\C^{n}})$. Then Theorem \ref{first rigidity} follows from Theorem \ref{simply connectedness}.
\end{proof}

\section{Maximal volume growth and second rigidity}\label{maximal volume growth and second rigidity section}

In this section, we will prove Theorem \ref{maximal volume growth} (i.e. maximal volume growth) and \ref{second rigidity} (i.e. second rigidity). As mentioned in Section \ref{introduction}, Theorem \ref{second rigidity} follows from Theorem \ref{maximal volume growth}. We will prove Theorem \ref{maximal volume growth} by refining Liu's argument of Theorem \ref{Liu result MVG}. Let us first recall Liu's argument in Subsection \ref{Liu argument}, and then give the proof of Theorem \ref{maximal volume growth} in Subsection \ref{maximal volume growth subsection}.

\subsection{Liu's argument}\label{Liu argument}

Suppose that $(M,\omega)$ does not have maximal volume growth. Let $X$ be a tangent cone at infinity. By the Cheeger-Colding theory \cite{CC97}, the Hausdorff dimension of $X$ is less than or equal to $2n-1$. Since the set of regular points is dense in $X$, then there is $q\in X$ such that the tangent cone at $q$ is isometric to $\R^{k}$ for some $k\leq 2n-1$. This implies that there exist $p_{i}\in M$ and $r_{i}>0$ such that $(M_{i},p_{i},\dist_{i},\nu_{i})$ converges to $(\R^{k},0,\dist_{\R^{k}},\nu)$ in the measured Gromov-Hausdorff sense, where
\[
M_{i} = M, \ \
\dist_{i} = \frac{\dist}{r_{i}}, \ \
\nu_{i} = \frac{\vol}{\vol(B_{r_{i}}(p_{i}))}
\]
and $\nu$ is proportional to the standard measure of $\R^{k}$. Recall some definitions in Subsection \ref{Some spaces}:
\[
\mathcal{H}_{d}(\R^{k}) = \{\text{complex harmonic polynomials on $\R^{k}$ with degree at most $d$}\}, \\[0.6mm]
\]
\[
\mathcal{P}_{d}(\C^{k-n}\times\R^{2n-k}) = \{\text{complex polynomials of $(z,x)$ with degree at most $d$}\}, \\[1.2mm]
\]
\[
\mathcal{E}_{d}(\C^{k-n}\times\R^{2n-k})
= \{f\in\mathcal{P}_{d}(\C^{k-n}\times\R^{2n-k}):\text{$f(z,\cdot)\in\mathcal{H}_{d}(\R^{2n-k})$ for any $z\in\C^{k-n}$}\}.
\]

\begin{lemma}[Lemma 2 and Lemma 4 of \cite{Liu16b}]\label{Liu lemmas}
Let $V_{i}$ be an $m$-dimensional subspace of $\O_{d}(M_{i})$. Then $V_{i}$ converges to an $m$-dimensional subspace $V$ of $\H_{d}(\R^{k})$ in the following sense. For any $f\in V$, there is a sequence $\{f_{i}\}$ such that $f_{i}\in\O_{d}(M_{i})$ and $f_{i}$ converges uniformly to $f$ in any compact set. Furthermore, if $n+1\leq k\leq 2n-1$, then there is a canonical identification $\R^{k}\cong\C^{k-n}\times\R^{2n-k}$ such that $V \subset \E_{d}(\C^{k-n}\times\R^{2n-k})$.
\end{lemma}

\medskip

It is clear that $\O_{d}(M_{i})=\O_{d}(M)$. Choosing $V_{i}=\O_{d}(M)$ and applying Lemma \ref{Liu lemmas}, there is a space $V$ of functions on $\R^{k}$ such that
\begin{itemize}\setlength{\itemsep}{1mm}
\item $\dim V=\dim\O_{d}(M)\geq cd^{n}$;
\item $V\subset\H_{d}(\R^{k})$ when $k\leq n$;
\item $V\subset\E_{d}(\C^{k-n}\times\R^{2n-k})$ when $n+1\leq k\leq 2n-1$.
\end{itemize}
According to the value of $k$, the argument can be divided into two cases.

\bigskip
\noindent
{\bf Case 1.} $k\leq n$.
\bigskip

By (i) of Remark \ref{dimension upper bound of E and H}, we compute
\[
cd^{n} \leq \dim V \leq \dim\H_{d}(\R^{k}) \leq C(k)d^{k-1} \leq C(n)d^{n-1},
\]
which is a contradiction.

\bigskip
\noindent
{\bf Case 2.} $n+1\leq k\leq 2n-1$.
\bigskip

By (ii) of Remark \ref{dimension upper bound of E and H}, we compute
\[
cd^{n} \leq \dim V \leq \dim\E_{d}(\C^{k-n}\times\R^{2n-k}) \leq C(n,k)d^{n-1},
\]
which is a contradiction.

\begin{remark}\label{Liu argument remark}
The above Liu's argument actually shows that the dimension condition
\begin{equation}\label{Liu argument remark eqn}
\begin{split}
\dim\O_{d}(M) \geq {} & \max_{k}\left(\dim\H_{d}(\R^{k}), \ \dim\E_{d}(\C^{k-n}\times\R^{2n-k})\right) \\[1mm]
= {} & \binom{n+d}{d}-\binom{n+d-2}{d-2}+1
\end{split}
\end{equation}
ensures the maximal volume growth. However, it is clear that
\[
\begin{split}
\binom{n+d}{d}-\binom{n+d-2}{d-2} - N_{3}
= {} & \binom{n+d}{d}-\binom{n+d-2}{d-2} - \binom{n-1+d}{d} \\
= {} & \binom{n-1+d}{d-1}-\binom{n+d-2}{d-2} \\
= {} & \binom{n+d-2}{d-1},
\end{split}
\]
and so \eqref{Liu argument remark eqn} is strictly stronger than $\dim\O_{d}(M)\geq N_{3}+1$ in general.
\end{remark}

\subsection{Proof of Theorem \ref{maximal volume growth}}\label{maximal volume growth subsection}

As mentioned in Remark \ref{Liu argument remark}, to prove Theorem \ref{maximal volume growth}, we need to refine Liu's argument. Let us begin with a straightforward observation.

\begin{lemma}\label{observation}
For any $f\in V$, the following holds:
\begin{itemize}\setlength{\itemsep}{1mm}
\item $f^{2}\in\H_{2d}(\R^{k})$ when $k\leq n$;
\item $f^{2}\in\E_{2d}(\C^{k-n}\times\R^{2n-k})$ when $n+1\leq k\leq 2n-1$.
\end{itemize}
\end{lemma}

\begin{proof}
Let $\{f_{i}\}$ be the sequence such that $f_{i}\in\O_{d}(M_{i})$ and $f_{i}$ converges uniformly to $f$ in any compact set. Then $\mathrm{span}\{f_{i}^{2}\}$ is an one-dimensional subspace of $\O_{d}(M_{i})$ and $\mathrm{span}\{f_{i}^{2}\}$ converges to $\mathrm{span}\{f^{2}\}$. By Lemma \ref{Liu lemmas}, the function $f^{2}$ satisfies the required property.
\end{proof}

Now we are ready to prove Theorem \ref{maximal volume growth}.

\begin{proof}[Proof of Theorem \ref{maximal volume growth}]
Following Liu's argument as in Subsection \ref{Liu argument}, it suffices to derive a contradiction when $k\leq 2n-1$. Recall that
\[
\dim V = \dim\O_{d}(M) \geq N_{3}+1 = \binom{n-1+d}{d}+1.
\]

\bigskip
\noindent
{\bf Case 1.} $k\leq n$.
\bigskip

In this case, $V$ is a subspace of $\H_{d}(\R^{k})$. Then $\ov{V}$ is also a subspace of $\H_{d}(\R^{k})$ with
\begin{equation}\label{case 1 dim ov V and V}
\dim \ov{V} = \dim V \geq \binom{n-1+d}{d}+1.
\end{equation}
Using \eqref{case 1 dim ov V and V} and Remark \ref{dimension of H}, we compute
\begin{equation}\label{dimension computation}
\begin{split}
\dim(V\cap\ov{V}) \geq {} & \dim V+\dim \ov{V}-\dim \H_{d}(\R^{k}) \\[4mm]
\geq {} & \dim V+\dim \ov{V}-\dim \H_{d}(\R^{n}) \\[3mm]
= {} & 2\binom{n-1+d}{d}+2-\binom{n+d}{d}+\binom{n+d-2}{d-2} \\[2mm]
= {} & \binom{n-1+d}{d}-\binom{n-1+d}{d-1}+\binom{n+d-2}{d-2}+2 \\[2mm]
= {} & \binom{n-1+d}{d}-\binom{n-2+d}{d-1}+2 \\[2mm]
= {} & \binom{n-2+d}{d}+2,
\end{split}
\end{equation}
which contradicts with Lemma \ref{constant lemma} below.

\begin{lemma}\label{constant lemma}
If $f\in V\cap\ov{V}$, then $f$ is constant.
\end{lemma}

\begin{proof}
Let $h$ be the real part of $f$, i.e.
\[
h = \frac{1}{2}(f+\ov{f}).
\]
Then $f\in V\cap\ov{V}$ implies that $h\in V$. By Lemma \ref{observation}, then $h^{2}$ is still harmonic. We compute
\[
2|\nabla h|^{2} = 2\nabla h \cdot \ov{\nabla h} = 2\nabla h \cdot \nabla h = \Delta h^{2}-2h\Delta h = 0,
\]
which implies $h$ is constant. Similarly, the imaginary part of $f$ is also constant.
\end{proof}

\bigskip
\noindent
{\bf Case 2a.} $k=2n-1$.
\bigskip

In this case, $V$ is a subspace of $\E_{d}(\C^{n-1}\times\R)$ and so any $f\in V$ can be written as
\[
f(z,x) = P_{0}(z)+P_{1}(z) \cdot x,
\]
where $P_{0}\in\O_{d}(\C^{n-1})$ and $P_{1}\in\O_{d-1}(\C^{n-1})$. It is clear that
\[
f^{2}(z,x) = P_{0}^{2}(z)+2P_{0}(z)P_{1}(z) \cdot x+P_{1}^{2}(z) \cdot x^{2}.
\]
Combining this with Lemma \ref{observation},
\[
0 = \Delta_{x}f^{2}(z,x) = 2P_{1}^{2}(z).
\]
It follows that $P_{1}\equiv0$ and so
\[
f(z,x) = P_{0}(z) \in \O_{d}(\C^{n-1}).
\]
Then we obtain $V\subset\O_{d}(\C^{n-1})$, which implies
\[
\dim V \leq \dim\O_{d}(\C^{n-1}) = \binom{n-1+d}{d} = N_{3}.
\]
This contradicts with $\dim V\geq N_{3}+1$.

\bigskip
\noindent
{\bf Case 2b.} $n+1\leq k\leq 2n-2$.
\bigskip

In this case, $V$ is a subspace of $\E_{d}(\C^{k-n}\times\R^{2n-k})$ and so any $f\in V$ can be written as
\[
f(z,x) = \sum_{|\alpha|\leq d}h_{\alpha}(x) \cdot z^{\alpha},
\]
where
\[
\alpha = (\alpha_{1},\ldots,\alpha_{k-n}), \ \
z^{\alpha} = z_{1}^{\alpha_{1}}\cdots z_{k-n}^{\alpha_{k-n}}, \ \
h_{\alpha} \in \H_{d-|\alpha|}(\R^{2n-k}).
\]
Define
\[
(Tf)(z,x) = \sum_{|\alpha|\leq d}\ov{h_{\alpha}(x)} \cdot z^{\alpha},
\]
and set
\[
TV = \{Tf: f\in V\}.
\]
Then $TV$ is a subspace of $\E_{d}(\C^{k-n}\times\R^{2n-k})$ with
\begin{equation}\label{case 2b dim TV and T}
\dim TV = \dim V \geq N_{3}+1 = \binom{n-1+d}{d}+1.
\end{equation}
By \eqref{case 2b dim TV and T}, Lemma \ref{dimension of E} and the similar computation of \eqref{dimension computation}, we obtain
\[
\begin{split}
\dim(V\cap TV) \geq {} & \dim V+\dim TV-\dim \E_{d}(\C^{k-n}\times\R^{2n-k}) \\[2mm]
\geq {} & 2\binom{n-1+d}{d}+2-\binom{n+d}{d}+\binom{n+d-2}{d-2} \\[1mm]
= {} & \binom{n-2+d}{d}+2.
\end{split}
\]
This implies
\begin{equation}\label{dim V TV}
\dim(V\cap TV) \geq \dim\O_{d}(\C^{n-2})+2.
\end{equation}

\smallskip

On the other hand, for any $f\in V\cap TV$, it follows from $T^{2}=\mathrm{Id}$ that $Tf\in V$ and so
\[
h = \frac{1}{2}(f+Tf) \in V.
\]
It is clear that $Th=h$. Applying Lemma \ref{independent of x} below, we obtain
\[
h = \frac{1}{2}(f+Tf) \in \O_{d}(\C^{k-n}).
\]
Similarly, we have
\[
\frac{1}{2\sqrt{-1}}(f-Tf) \in \O_{d}(\C^{k-n}).
\]
It follows that $f\in\O_{d}(\C^{k-n})$. Then we obtain
\[
V\cap TV \subset \O_{d}(\C^{k-n}).
\]
From $n+1\leq k\leq 2n-2$, we obtain $1\leq k-n\leq n-2$ and then
\[
\dim(V\cap TV) \leq \dim\O_{d}(\C^{k-n}) \leq \dim\O_{d}(\C^{n-2}),
\]
which contradicts with \eqref{dim V TV}.

\begin{lemma}\label{independent of x}
For any $h\in V$, if $Th=h$, then $h\in\O_{d}(\C^{k-n})$.
\end{lemma}

\begin{proof}
Set $p=k-n$ and $q=2n-k$. To prove Lemma \ref{independent of x}, by Lemma \ref{observation}, it suffices to prove the following claim: for any
\[
h(z,x) = \sum_{|\alpha|\leq d}h_{\alpha}(x) \cdot z^{\alpha} \in \E_{d}(\C^{p}\times\R^{q}),
\]
if $h$ satisfies
\begin{itemize}\setlength{\itemsep}{1mm}
\item $h^{2}\in\E_{2d}(\C^{p}\times\R^{q})$;
\item $h_{\alpha}(x)$ is real for each $\alpha$,
\end{itemize}
then $h$ is independent of $x$, i.e. $h(z,x)=h(z)\in\O_{d}(\C^{p})$.

\medskip

We will prove this claim by induction on $p$.

\bigskip
\noindent
{\bf Step 1.} $p=1$.
\bigskip

We decompose
\[
h(z_{1},x) = \sum_{\alpha_{1}=0}^{d}h_{\alpha_{1}}(x) \cdot z_{1}^{\alpha_{1}} \in \E_{d}(\C\times\R^{q})
\]
and
\[
h^{2}(z_{1},x) = \sum_{\ell=0}^{2d} \left(\sum_{\alpha_{1}+\beta_{1}=\ell} h_{\alpha_{1}}(x) \cdot h_{\beta_{1}}(x) \right) z_{1}^{\ell} \in \E_{2d}(\C\times\R^{q}).
\]
It then follows that
\begin{equation}\label{independent of x case 1 eqn 1}
h_{\alpha_{1}}(x) \in \H_{d-\alpha_{1}}(\R^{q}) \ \text{ for} \ \alpha_{1} = 0,\ldots,d
\end{equation}
and
\begin{equation}\label{independent of x case 1 eqn 2}
\sum_{\alpha_{1}+\beta_{1} = \ell} h_{\alpha_{1}}(x) \cdot h_{\beta_{1}}(x) \in \H_{2d-\ell}(\R^{q}) \ \text{for} \ \ell=0,\ldots,2d.
\end{equation}
To prove $h$ is independent of $x$, it suffices to show that each $h_{\alpha_{1}}$ is constant. From \eqref{independent of x case 1 eqn 1}, we obtain $h_{d}(x)\in\H_{0}(\R^{q})$ and so $h_{d}$ is constant. Taking $\ell=2d-2$ in \eqref{independent of x case 1 eqn 2}, we see that
\begin{equation}\label{independent of x case 1 eqn 3}
h_{d-1}^{2}(x) + 2h_{d}(x) \cdot h_{d-2}(x) \in \H_{2}(\R^{q}).
\end{equation}
By \eqref{independent of x case 1 eqn 1}, we have $h_{d-2}(x)\in\H_{2}(\R^{q})$. Since $h_{d}$ is constant, then
\[
2h_{d}(x) \cdot h_{d-2}(x) \in \H_{2}(\R^{q}).
\]
Combining this with \eqref{independent of x case 1 eqn 3}, we see that $h_{d-1}^{2}(x) \in \H_{2}(\R^{q})$. Using \eqref{independent of x case 1 eqn 1} again, we have $h_{d-1}(x)\in\H_{1}(\R^{q})$. Since $h_{d-1}$ is real, then
\[
\begin{split}
2|\nabla_{x}h_{d-1}|^{2} = {} & 2\nabla_{x}h_{d-1}\cdot \ov{\nabla_{x}h_{d-1}} = 2\nabla_{x}h_{d-1}\cdot \nabla_{x}h_{d-1} \\
= {} & \Delta_{x}(h_{d-1}^{2}) - 2h_{d-1}\Delta_{x}h_{d-1} = 0,
\end{split}
\]
which implies $h_{d-1}$ is constant. Taking $\ell=2d-4,\ldots,0$ in \eqref{independent of x case 1 eqn 2} and repeating the above process, we obtain that each $h_{\alpha_{1}}$ is constant.

\bigskip
\noindent
{\bf Step 2.} $p\geq2$.
\bigskip

Suppose by induction that the claim is true for $p-1$. Write $z=(z_{1},z')$ and $\alpha=(\alpha_{1},\alpha')$ where
\[
z' = (z_{2},\ldots,z_{q}), \ \ \alpha' = (\alpha_{2},\ldots,\alpha_{q}).
\]
For any
\[
h(z,x) = \sum_{|\alpha|\leq d}h_{\alpha}(x) \cdot z^{\alpha} \in \E_{d}(\C^{p}\times\R^{q}),
\]
we rearrange the above terms as
\[
h(z_{1},z',x) = \sum_{\alpha_{1}=0}^{d}\left(\sum_{|\alpha'|\leq d-\alpha_{1}}h_{(\alpha_{1},\alpha')}(x)\cdot (z')^{\alpha'}\right) z_{1}^{\alpha_{1}}.
\]
Set
\[
g_{\alpha_{1}}(z',x) = \sum_{|\alpha'|\leq d-\alpha_{1}}h_{(\alpha_{1},\alpha')}(x)\cdot (z')^{\alpha'}
\]
and then
\begin{equation}\label{independent of x case 2 eqn 1}
g_{\alpha_{1}}(z',x) \in \E_{d-\alpha_{1}}(\C^{p-1}\times\R^{q}) \ \text{for} \ \alpha_{1}=0,\ldots,d.
\end{equation}
We aim to show that each $h_{\alpha_{1}}$ is independent of $x$. From
\[
h^{2}(z_{1},z',x) = \sum_{\ell=0}^{2d} \left(\sum_{\alpha_{1}+\beta_{1}=\ell} g_{\alpha_{1}}(z',x) \cdot g_{\beta_{1}}(z',x) \right) z_{1}^{\ell} \in \E_{2d}(\C^{p}\times\R^{q}),
\]
we obtain
\begin{equation}\label{independent of x case 2 eqn 2}
\sum_{\alpha_{1}+\beta_{1}=\ell} g_{\alpha_{1}}(z',x) \cdot g_{\beta_{1}}(z',x) \in \E_{2d-\ell}(\C^{p-1}\times\R^{q}) \ \text{for} \ \ell=0,\ldots,2d.
\end{equation}
By \eqref{independent of x case 2 eqn 1}, we obtain $g_{d}$ is independent of $x$ (actually $g_{d}$ is constant) and $g_{d-2}(z',x)\in\E_{2}(\C^{p-1}\times\R^{q})$. Then
\[
2g_{d}(z',x) \cdot g_{d-2}(z',x) = 2g_{d}(z') \cdot g_{d-2}(z',x) \in \E_{2}(\C^{p-1}\times\R^{q}).
\]
Taking $\ell=2d-2$ in \eqref{independent of x case 2 eqn 2} and using the similar argument of Step 1, we have
\[
g_{d-1}^{2}(z',x) \in \E_{2}(\C^{p-1}\times\R^{q}).
\]
Using \eqref{independent of x case 2 eqn 1} again, we have $g_{d-1}(z',x) \in \E_{1}(\C^{p-1}\times\R^{q})$. It is clear that $h_{(\alpha_{1},\alpha')}(x)$ is real for each $\alpha'$. By the inductive hypothesis, we obtain that $g_{d-1}$ is independent of $x$. Taking $\ell=2d-4,\ldots,0$ in \eqref{independent of x case 2 eqn 2} and repeating the above process, we obtain that each $g_{\alpha_{1}}$ is independent of $x$.
\end{proof}

\end{proof}

\section{Examples}\label{examples}

In this section, some examples are given, which indicate that the main results in this paper are optimal. For any $\lambda\in(0,1)$, consider the K\"ahler metric on $\C$:
\[
\omega_{\C,\lambda} = (1+|z|^{2})^{\lambda-1}\cdot\omega_{\C}.
\]
By straightforward calculation, it can be checked that
\begin{itemize}\setlength{\itemsep}{1mm}
\item $\omega_{\C,\lambda}$ is rotationally symmetric;
\item $\omega_{\C,\lambda}$ has positive holomorphic bisectional curvature;
\item the $\omega_{\C,\lambda}$-geodesic distance from the origin is asymptotic to $\lambda^{-1}\cdot|z|^{\lambda}$ as $|z|\to\infty$.
\end{itemize}

\begin{proposition}
Let $d\geq1$ be an integer. When $\lambda\in(\frac{d-1}{d},1)$, the K\"ahler manifold
\[
(\C^{n},\omega_{\lambda})=(\mathbb{C}\times\mathbb{C}^{n-1},\omega_{\mathbb{C},\lambda}+\omega_{\mathbb{C}^{n-1}})
\]
satisfies the following properties.
\begin{enumerate}\setlength{\itemsep}{1mm}
\item[(i)] $(\C^{n},\omega_{\lambda})$ has non-negative holomorphic bisectional curvature.
\item[(ii)] The dimension of space $\mathcal{O}_{d}(\C^{n},\omega_{\lambda})$ is
\[
\dim\mathcal{O}_{d}(\C^{n},\omega_{\lambda}) = N_{2} = \dim\mathcal{O}_{d}(\mathbb{C}^{n})-\binom{n+d-2}{d-1}.
\]
\end{enumerate}
\end{proposition}

\begin{proof}
(i) is trivial. For (ii), write $z=(z_{1},z')$ where $z'=(z_{2},\dots,z_{n})$. For any $f\in\mathcal{O}_{d}(\C^{n},\omega_{\lambda})$, we have
\begin{equation}\label{example eqn 1}
|f(z_{1},z')| \leq C(1+|z_{1}|^{\lambda d}+|z'|^{d}).
\end{equation}
Since $\lambda<1$, then $|f(z)|\leq C(1+|z|)^{d}$. The classical Liouville theorem (i.e. Theorem \ref{classical Liouville theorem}) implies that $f$ must be a polynomial with degree at most $d$. We use $V$ to denote the space spanned by
\begin{itemize}\setlength{\itemsep}{1mm}
\item monomial of $z=(z_{1},z')$ with degree at most $d-1$;
\item monomial of $z'$ with degree $d$.
\end{itemize}
It is clear that
\[
\begin{split}
\dim V = {} & \dim\mathcal{O}_{d-1}(\mathbb{C}^{n})+\binom{n+d-2}{d} \\
= {} & \binom{n+d-1}{d-1}+\binom{n+d-2}{d} \\[0.2mm]
= {} & \binom{n+d}{d}-\binom{n+d-2}{d-1} \\
= {} & \dim\mathcal{O}_{d}(\mathbb{C}^{n})-\binom{n+d-2}{d-1}.
\end{split}
\]
To prove (ii), it suffices to show $V=\mathcal{O}_{d}(\mathbb{C}^{n},\omega_{\lambda})$. We split the argument into two steps.

\bigskip
\noindent
{\bf Step 1.} $V\subseteq\mathcal{O}_{d}(\mathbb{C}^{n},\omega_{\lambda})$.
\bigskip

Since $\lambda>\frac{d-1}{d}$, then $\lambda d>d-1$ and so for any $f\in V$,
\[
|f(z_{1},z')| \leq C(1+|z_{1}|^{d-1}+|z'|^{d})
\leq C(1+|z_{1}|^{\lambda d}+|z'|^{d}),
\]
which implies $f\in\mathcal{O}_{d}(\mathbb{C}^{n},\omega_{\lambda})$.

\bigskip
\noindent
{\bf Step 2.} $\mathcal{O}_{d}(\mathbb{C}^{n},\omega_{\lambda})\subseteq V$.
\bigskip

We argue by contradiction. Suppose that $\mathcal{O}_{d}(\mathbb{C}^{n},\omega_{\lambda})\nsubseteq V$, then there exists $f\in\mathcal{O}_{d}(\mathbb{C}^{n},\omega_{\lambda})$ such that
\begin{itemize}\setlength{\itemsep}{1mm}
\item[(a)] $f$ is a homogeneous polynomial with degree $d$;
\item[(b)] $f$ involves $z_{1}$;
\end{itemize}
Then $f$ can be written as
\[
f(z_{1},z') = \sum_{k=0}^{m}f_{k}(z')\cdot z_{1}^{k},
\]
where $m\leq d$, each $f_{k}(z')$ is a polynomial and $f_{m}(z')$ is a non-zero polynomial. From (a), we obtain that, for each $k$, either $f_{k}(z')\equiv0$, or $f_{k}(z')$ is a homogenous polynomial with degree $d-k$. In both cases, we have $f_{k}(\theta z')=\theta^{d-k}f_{k}(z')$ for any $\theta\in\mathbb{R}$. From (b), $f$ involves $z_{1}$ and so $m\geq1$. Now we fix $z_{0}'\in\C^{n-1}$ such that $f_{m}(z_{0}')\neq0$. For any $\mu\in\R$, we have
\[
f(\mu,\mu^{\lambda}\cdot z_{0}') = \sum_{k=0}^{m}f_{k}(\mu^{\lambda}\cdot z_{0}')\cdot \mu^{k}
= \sum_{k=0}^{m}f_{k}(z_{0}')\cdot\mu^{\lambda(d-k)+k}
= \sum_{k=0}^{m}f_{k}(z_{0}')\cdot\mu^{\lambda d+(1-\lambda)k}.
\]
Since $\lambda<1$, then $f_{m}(z_{0}')\cdot\mu^{\lambda d+(1-\lambda)m}$ is the dominated term and so
\begin{equation}\label{example eqn 2}
\lim_{\mu\to\infty}\Bigg|\frac{f(\mu,\mu^{\lambda}\cdot z_{0}')}{f_{m}(z_{0}')\cdot\mu^{\lambda d+(1-\lambda)m}}\Bigg| = 1.
\end{equation}
On the other hand, \eqref{example eqn 1} implies
\[
|f(\mu,\mu^{\lambda}\cdot z_{0}')| \leq C(1+|\mu|^{\lambda d}+|\mu^{\lambda}\cdot z_{0}'|^{d})
\leq C(1+|z_{0}'|^{d})\cdot\mu^{\lambda d}.
\]
Combining this with \eqref{example eqn 2}, we obtain
\[
1 = \lim_{\mu\to\infty}\Bigg|\frac{f(\mu,\mu^{\lambda}\cdot z_{0}')}{f_{m}(z_{0}')\cdot\mu^{\lambda d+(1-\lambda)m}}\Bigg|
\leq \lim_{\mu\to\infty}\frac{C(1+|z_{0}'|^{d})\cdot\mu^{\lambda d}}{|f_{m}(z_{0}')|\cdot\mu^{\lambda d+(1-\lambda)m}}
= 0,
\]
which is a contradiction.
\end{proof}

\section{Further questions}\label{questions section}

As mentioned in Section \ref{introduction}, Liu \cite{Liu16a} gave an alternative proof of Theorem \ref{dimension thm} under weaker curvature assumption. Precisely, Liu \cite{Liu16a} proved Theorem \ref{dimension thm} when the holomorphic sectional curvature is non-negative.

\begin{definition}
Let $(M,\omega)$ be an $n$-dimensional K\"ahler manifold and $K$ be a constant. We say that the holomorphic sectional curvature is greater than or equal to $K$ if
\[
R(u,\ov{u},u,\ov{u}) \geq K|u|^{4}
\]
for any $(1,0)$-vector $u$.
\end{definition}

\begin{theorem}[Liu \cite{Liu16a}]\label{Liu result}
Let $(M,\omega)$ be a complete $n$-dimensional K\"ahler manifold with non-negative holomorphic sectional curvature. For any integer $d\geq1$,
\[
\dim\mathcal{O}_{d}(M) \leq \dim\mathcal{O}_{d}(\mathbb{C}^{n}),
\]
and the equality holds if and only if $(M,\omega)$ is biholomorphically isometric to $(\mathbb{C}^{n},\omega_{\C^{n}})$.
\end{theorem}

Motivated by Theorem \ref{Liu result}, it is natural to expect that the main results in this paper still hold under non-negative holomorphic sectional curvature assumption.

\begin{question}
Let $(M,\omega)$ be a complete $n$-dimensional K\"ahler manifold with non-negative holomorphic sectional curvature. If there exists an integer $d\geq1$ such that
\[
\dim\mathcal{O}_{d}(M) \geq N_{2}+1 = \dim\mathcal{O}_{d}(\mathbb{C}^{n})-\binom{n+d-2}{d-1}+1,
\]
then $(M,\omega)$ is biholomorphically isometric to $(\mathbb{C}^{n},\omega_{\C^{n}})$.
\end{question}

\begin{question}
Let $(M,\omega)$ be a complete $n$-dimensional K\"ahler manifold with non-negative holomorphic sectional curvature. If there exists an integer $d\geq1$ such that
\[
\dim\mathcal{O}_{d}(M) \geq N_{3}+1 = \binom{n-1+d}{d}+1,
\]
then $(M,\omega)$ has maximal volume growth.
\end{question}

\begin{question}
Let $(M,\omega)$ be a complete $n$-dimensional K\"ahler manifold with non-negative holomorphic sectional curvature. If there exists an integer $d\geq1$ such that
\[
\dim\mathcal{O}_{d}(M) \geq N_{3}+1 = \binom{n-1+d}{d}+1,
\]
then $M$ is biholomorphic to $\mathbb{C}^{n}$.
\end{question}

\end{document}